\documentclass{amsart}

\usepackage{amssymb,amscd,amsthm,amsxtra}
\usepackage{dsfont,latexsym}
\usepackage{url}

\usepackage{mathrsfs}
\renewcommand{\mathcal}{\mathscr}

\usepackage{color}
\definecolor{citation}{rgb}{0.2,0.5,0.2}
\definecolor{formula}{rgb}{0.1,0.2,0.5}
\definecolor{url}{rgb}{0,0.2,0.7}
\usepackage[colorlinks=true,linkcolor=formula,urlcolor=url,citecolor=citation]{hyperref}

\newtheorem{thm}{Theorem}[section]

\newtheorem{lemma}[thm]{Lemma}
\newtheorem{prop}[thm]{Proposition}
\theoremstyle{definition}

\theoremstyle{remark}
\newtheorem{rem}[thm]{Remark}
\numberwithin{equation}{section}

\newcommand{\CC}{\mathcal{C}}

\newcommand{\FF}{\mathscr{F}}

\newcommand{\EE}{\mathcal{E}}
\newcommand{\SF}{{\mathscr S}}

\newcommand{\R}{{\mathds R}}
\newcommand{\N}{{\mathds N}}

\def\dys{\displaystyle}
\def\rif{\eqref}
\def\eps{\varepsilon}

\newlength{\defbaselineskip}
\newcommand{\setlinespacing}[1]
           {\setlength{\baselineskip}{#1 \defbaselineskip}}

\title
[Existence and symmetry results for a Schr\"odinger type problem]
{Existence and symmetry results \\ for a Schr\"odinger type problem \\ involving the fractional Laplacian}

\author[S. Dipierro]{Serena Dipierro}
\address[Serena Dipierro]{SISSA, Sector of Mathematical Analysis
\\ Via Bonomea, 265 
\\ 34136 Trieste, Italy}
\email{\href{mailto:dipierro@sissa.it}{dipierro@sissa.it}}

\author[G. Palatucci]{Giampiero Palatucci}
\address[Giampiero Palatucci]{Dipartimento di Matematica
\\ Universit\`a degli Studi di Parma
\\ Campus - Parco Area delle Scienze, 53/A
\\ 43124 Parma, Italy}
 \email{\href{mailto:giampiero.palatucci@unipr.it}{giampiero.palatucci@unipr.it}}

\author[E. Valdinoci]{Enrico Valdinoci}
\address[Enrico Valdinoci]{Dipartimento di Matematica
\\ Universit\`a degli Studi di Milano
\\ Via Saldini, 50
\\ 20133 Milano, Italy}
\email{\href{mailto:enrico@math.utexas.edu}{enrico@math.utexas.edu}}

\begin{document}
\vskip .2truecm

\subjclass[2010]{Primary  35J60, 35B33; Seconday 35S30, 49J45\vspace{1mm}}

\keywords{Nonlinear problems, fractional Laplacian, fractional Sobolev spaces, critical Sobolev exponent, spherical solutions, ground states. \vspace{1mm}}

\thanks{{\it Acknowledgments}. The second author has been supported by
\href{http://prmat.math.unipr.it/~rivista/eventi/2010/ERC-VP/}{ERC grant 207573 ``Vectorial Problems''}. 
The third author has been supported by FIRB ``Project Analysis and Beyond'' and by ERC grant 277749 ``$\epsilon$ Elliptic Pde's and Symmetry of Interfaces and Layers for Odd Nonlinearities''.}

\begin{abstract}
\small{This paper deals with the following class of nonlocal Schr\"odinger equations
$$
\displaystyle
(-\Delta)^s u +  u =  |u|^{p-1}u \ \  \text{in} \ \R^N, \quad \text{for} \ s\in (0,1).
$$
We prove existence and symmetry results for the solutions $u$ in the fractional Sobolev space $H^s(\R^N)$. Our results are in clear accordance with those for the classical local counterpart, that is when $s=1$.
}
\end{abstract}

\setcounter{tocdepth}{2}

\maketitle
{\small \tableofcontents}

\setlinespacing{1.09}

\section{Introduction}

We consider the following problem
\begin{eqnarray}\label{pb_classic}
\dys
\begin{cases}
-\Delta u + \eta u = \lambda |u|^{p-1}u & \text{in} \ \R^N, \\[1ex]
u\in H^1(\R^N), \ u \not\equiv 0,
\end{cases}
\end{eqnarray}
where $\lambda$ and $\eta$ are fixed positive constants and $p>1$.

The equation in~\rif{pb_classic} has been widely studied in the last decades, since it is the basic version of some
fundamental models arising in various applications (e.~\!g., stationary states in nonlinear equations of Schr\"odinger type). One of the first contributions to the analysis of problem~\rif{pb_classic} was given by Pohozaev in~\cite{Poh65}, where he proved that there exists a solution~$u$ of~\rif{pb_classic} if and only if $1<p<2^*\!-\!1$, being $2^*=2N/(N-2)$ the so-called Sobolev critical exponent. In~\cite{Poh65} also a by-now classical ``identity'' appears, in order to prove that there are no solutions to~\rif{pb_classic} when $p$ is greater or equal than $2^*\!-1$.

Another important contribution to the analysis of 
problem~\rif{pb_classic} has been given in~\cite{BL83}
(see also \cite{BL83b}), in which the authors consider 
an extension of the equation in~\rif{pb_classic} by replacing the 
nonlinearity $-\eta u+\lambda|u|^{p-1}u$ by a wider class of odd 
continuous functions~$g=g(u)$ satisfying $g(0)=0$ and some superlinear 
and growth assumptions. Among other results, in~\cite{BL83} it has been 
shown the existence of a solution~$u$ to~\rif{pb_classic}, with some 
properties of symmetry and a precise decay at infinity. It is worth 
pointing out that the method to prove the existence of solutions 
to~\rif{pb_classic} relies on a variational approach \big(the {\it 
constrained minimization method}, see~\cite[Section 3]{BL83}\big), by 
working directly with the energy functional related to~\rif{pb_classic}.

\vspace{1mm}
A natural question could be whether or not this method can be adapted to deal 
with a nonlocal version of the problem above. In this respect, the aim of the present paper is to extend the existence and symmetry results in~\cite{BL83} for the nonlocal analog of problem~\rif{pb_classic} by replacing the standard Laplacian operator by the fractional Laplacian operator~$(-\Delta)^s$, where, as usual, for any $s\in(0,1)$, $(-\Delta)^s$ denotes the $s$-power of the Laplacian operator and, omitting a multiplicative constant $C=C(N,s)$, we have
\begin{equation}\label{def_fractio}
\dys (-\Delta)^su(x)\, =\, P.V.\int_{\R^N}\frac{u(x)-u(y)}{|x-y|^{N+2s}}\,dy \, = \, \lim_{\eps\to 0} \int_{\CC B_\eps(x)}\frac{u(x)-u(y)}{|x-y|^{N+2s}}\,dy.
\end{equation}
Here $B_\eps(x)$ denotes the $N$-dimensional ball of radius~$\eps$, centered at~$x\in\R^N$, $\CC$~denotes the complementary set,
and 
``$P.V.$'' is a commonly used abbreviation for ``in the principal value sense''. 
\vspace{1mm}

Recently, a great attention has been focused on the study of problems involving the fractional Laplacian,  from a pure mathematical point of view 
as well as from concrete applications, since this operator  naturally arises in many different contexts, such as, among the others, obstacle problems, financial market, phase transitions, anomalous diffusions, crystal dislocations, soft thin films, semipermeable membranes, flame propagations, conservation laws, ultra-relativistic limits of quantum mechanics, quasi-geostrophic flows, minimal surfaces, materials science, water waves, etc...
The literature is really too wide to attempt any reasonable comprehensive treatment in a single paper\footnote{For an elementary introduction to this topic and a wide, but still not fully comprehensive, list of related references, we refer to~\cite{DPV12}.}.
We would just cite some very recent papers which analyze fractional elliptic equations involving the critical Sobolev exponent, \cite{SV12, Tan12, FL10, BCD10, BCD11, PP11, SV11}.

\vspace{1mm}

Let us come back to the present paper. We will deal with the following problem
\begin{eqnarray}\label{problema}
\dys
\begin{cases}
(-\Delta)^s u +  u =  |u|^{p-1}u \ \  \text{in} \ \R^N, \\[1ex]  
u\in H^s(\R^N), \ u \not\equiv 0,
\end{cases}
\end{eqnarray}
where $H^s(\R^N)$ denotes the fractional Sobolev space; we immediately refer to Section~\ref{sec_fractio} for the definitions of the space $H^s(\R^N)$ 
and of variational solutions to~\rif{problema}.

Precisely, we are interested in existence and symmetry properties of the variational solutions~$u$ to~\rif{problema}, as stated in the following
\begin{thm}\label{teorema} 
Let $s\in (0,1)$ and $p \in (1,\, (N+2s)/(N-2s))$, with $N\geq 2$. There exists a solution $u\in H^s(\R^N)$ to problem~\rif{problema} which is positive and spherically symmetric. 
\end{thm}
Note that the upperbound on the exponent $p$ is exactly $2^\ast_s+1$, where $2^*_s=2N/(N-2s)$ is the critical Sobolev exponent of the embedding $H^s\hookrightarrow L^p$.
This fractional Sobolev exponent also plays a role
for the nonlinear analysis methods for equations in bounded
domains; see \cite{SV11}.
\vspace{1mm}

The proof of Theorem~\ref{teorema} extends part of that of Theorem 2 in~\cite{BL83}; in particular, we will apply the variational approach by the constrained method mentioned above, for the 
energy functional related
to~\rif{problema}, that is
\begin{equation}\label{def_energia1}
\dys
\EE(u)\,:=\, \frac{1}{2}\iint_{\R^N\times \R^N} \frac{|u(x)-u(y)|^2}{|x-y|^{N+2s}}\,dx\,dy + \int_{\R^N}\Big(\frac{1}{2}|u(x)|^2- \frac{1}{p+1}|u(x)|^{p+1}\Big)\, dx.
\end{equation}
\vspace{1mm}

It is worth noticing that, although the general strategy of the proof will follow the original argument in~\cite{BL83}, we need to operate various modifications due to the non-locality of the fractional Laplacian operator \big(and of the correspondent  norm $H^s(\R^N)$\big). 
Moreover, we will need some energy estimates and preliminary results, 
also including the analog of the classical Polya-Szeg\"o inequality, as 
given in forthcoming Section~\ref{sec_tools}. 

\vspace{1mm}

As for the precise decay of the solution found, a precise bound
may be obtained
via the construction of exact barriers (see Lemma 3.1 in~\cite{SV11x} and, also, Lemma 8 in~\cite{PSV12}). 
Also, it could be taken into account to extend all the results above in 
order to investigate a problem of type~\rif{problema} by substituting 
the nonlinearity with an odd continuous function satisfying standard 
growth assumptions, in the same spirit of~\cite{BL83}.

\vspace{2mm}
The paper is organized as follows. In Section~\ref{sec_prelim} below, we fix notation and we state and prove some preliminary results. Section~\ref{sec_main} is devoted to the proof of Theorem~\ref{teorema}.

\vspace{2mm}
\section{Preliminary results}\label{sec_prelim}

In this section, we state and prove a few preliminary results that we will need in the rest of the paper. First, we will recall some definitions involving the fractional Laplacian operator and we give the definition of the solutions to the problem we are dealing with.

\subsection{Notation}\label{sec_notation}
In the present paper we follow the usual convention of denoting by~$C$ a general positive constant, possibly varying from line to line. Relevant dependencies on parameters will be emphasized by using parentheses; special constants will be denoted by $C_1, C_2$, ...
\vspace{2mm}

We consider the Schwartz space $\SF$ of rapidly decaying $C^{\infty}$ functions in $\R^N$, with the corresponding topology generated by the seminorms
$$
p_N(\varphi)= \sup_{x\in\R^N} (1+|x|)^N \sum_{|\alpha|\leq N} |D^{\alpha}\varphi(x)|\,, \quad N=0,1,2,...\,, 
$$
where $\varphi\in \SF(\R^N)$. 
Let $\SF'(\R^N)$ be the set of all tempered distributions, that is the topological dual of $\SF(\R^N)$. As usual, for any $\varphi\in \SF(\R^N)$, we denote by
$$
\FF\varphi(\xi)\, = \, \frac{1}{(2\pi)^{{N}/{2}}} \! \int_{\R^N} e^{-i\xi \cdot x}\,\varphi(x)\,dx
$$
the Fourier transform of $\varphi$ and we recall that one can extend $\FF$ from $\SF(\R^N)$ to $\SF'(\R^N)$.

\vspace{2mm}
For any $s\in (0,1)$, the fractional Sobolev space $H^s(\R^N)$ is defined by
\begin{equation}\label{def_hs}
H^{s}(\R^N)= \left\{ u\in L^2(\R^N)\; :\; \frac{|u(x)-u(y)|}{|x-y|^{\frac{N}{2}+s}} \in L^2(\R^N \times \R^N)  \right\},
\end{equation}
endowed with the natural norm
\begin{equation*}
\|u\|_{H^{s}(\Omega)}= \left( \int_{\R^N}|u|^2 \,dx\,+\,\iint_{\R^N\times\R^N} \frac{|u(x)-u(y)|^2}{|x-y|^{N+2s}} dx\, dy  \right)^{\!\frac{1}{2}}\!,
\end{equation*}
where the term
\begin{equation}\label{def_hs0}
\dys 
[u]_{H^{s}(\R^N)} \, = \, \|(-\Delta)^{\frac{s}{2}} u\|_{L^2(\R^N)}
\, := \, \left( \iint_{\R^N} \frac{|u(x)-u(y)|^2}{|x-y|^{N+2s}} dx\, dy  \right)^{\!\frac{1}{2}}
\end{equation}
is the so-called {\it Gagliardo semi-norm} of $u$.

\vspace{1mm}
\subsection{A few basic results on the fractional Laplacian and setting of the problem}\label{sec_fractio}
In the following, we make use of equivalent definitions of the fractional Laplacian and  the Gagliardo semi-norm via the Fourier transform.
Indeed, the fractional Laplacian $(-\Delta)^s$ can be seen as a pseudo-differential operator of symbol $|\xi|^s$, as stated in the following

\begin{prop}\label{prop_33}{\rm (see, e.g.,}
\cite[Proposition 3.3]{DPV12} {\rm or}~\cite[Section 3]{Val09}{\rm )}.
Let $s \in (0,1)$ and let $(-\Delta)^s:\SF\to L^2(\R^N)$ be the fractional operator defined by~\rif{def_fractio}. Then, for any $u\in\SF$,
$$
(-\Delta)^s u = \FF^{-1}(|\xi|^{2s}(\FF u)) \ \ \ \forall \xi \in \R^N,
$$
up to a multiplicative constant.
\end{prop}
Analogously, one can see that the fractional Sobolev space $H^s(\R^N)$, given by~\rif{def_hs}, can be defined via the Fourier transform as follows
\begin{equation}\label{def_viafourier}
H^{s}(\R^N)=\left\{\,u\in L^2(\R^N)\;\, :\;\,\int_{\R^N} (1+|\xi|^{2s}) |\FF u(\xi)|^2\,d\xi < +\infty \, \right\}.
\end{equation}
This is a natural consequence of the equivalence stated in the following proposition, whose proof relies on the Plancherel formula.
\begin{prop}\label{pro_equiv}{\rm (see, e.g.,}
\cite[Proposition 3.4]{DPV12}{\rm )}.
Let $s\in (0,1)$. For any $u\in H^s(\R^N)$
\begin{equation}\label{def_viaf}
[u]^2_{H^s(\R^N)} \, = \,  \int_{\R^N} |\xi|^{2s} |\FF u(\xi)|^2\, d\xi,
\end{equation}
up to a multiplicative constant.
\end{prop}

Finally, we recall the definition of variational solutions~$u\in H^s(\R^N)$ to
\begin{equation}\label{pb}
(-\Delta)^s u + u = |u|^{p-1}u \quad \text{in} \ \R^N, \quad u\not\equiv 0,
\end{equation}
where $p>1$.

For any $s\in(0,1)$, a measurable function $u:\R^N\to \R$ is a variational solution to~\eqref{pb} if
\begin{eqnarray}\label{def_sol}
\dys
&& \iint_{\R^N\times \R^N} \frac{\big(u(x)-u(y)\big)\big(\varphi(x)-\varphi(y)\big)}{|x-y|^{N+2s}}\,dx\,dy + \int_{\R^N}u(x)\varphi(x)\,dx \nonumber \\
&& \qquad \qquad \qquad \qquad  \qquad \qquad \qquad \qquad  \qquad \qquad  \quad
= \int_{\R^N}|u(x)|^{p-1}u(x)\varphi(x)\,dx,
\end{eqnarray}
for any function $\varphi \in C^1_0(\R^N)$.

As stated in the Introduction, a natural method to solve~\eqref{pb} is to look for critical points of the related energy functional~$\EE$ on the space~$H^s(\R^N)$ defined in~\eqref{def_energia1}, that is
\begin{equation}\label{def_energia}
\dys
\EE(u)\,:=\,\frac{1}{2} [u]^2_{H^s(\R^N)}  - \int_{\R^N}G(u)\, dx,
\end{equation}
where $[u]_{H^s}$ is defined by~\rif{def_hs0} and we denoted by $G$ the function 
\begin{equation}\label{def_g}
G(u):= \frac{1}{p+1}|u|^{p+1} - \frac{1}{2}|u|^2.
\end{equation}
Therefore, from now on we will focus on the following variational problem 
\begin{equation}\label{var}
\dys
\min \left\{ [u]^2_{H^s(\R^N)} : u\in H^s(\R^N),\, \int_{\R^N} G(u)\, dx = 1 \right\}.
\end{equation}

\vspace{1mm}
\subsection{Tools}\label{sec_tools}

As already mentioned, \cite{Poh65} provided an elementary identity from which one can deduce some necessary conditions for the existence of a solution to problem~\rif{pb_classic}. Analogously, a solution to problem~\rif{pb} has to satisfy a Pohozaev identity for any $s\in(0,1)$, that is of type
\begin{equation}\label{eq_spoho} 
\frac{N-2s}{2}\,[u]^2_{H^s(\R^N)} \, = \, N\int_{\R^N} G(u)\, dx,
\end{equation}
where $G$ is given by~\rif{def_g}. In view of the definition of the fractional norm via the Fourier transform in~\rif{def_viaf}, a proof can be obtained by modifying the general arguments in~\cite{Poh65}, that is, by choosing suitable test functions (see, e.~\!g., Lemma~5.1 in~\cite{FL10}, where properties of ground state solutions for the equation~\rif{eq_spoho} in 1D are investigated).

\vspace{2mm}
Now, for any measurable function $u$ consider the corresponding symmetric radial decreasing rearrangement $u^*$, whose classical
definition and basic properties
can be found, for instance, in~\cite[Chapter 2]{Kaw85}. 
As in the classic case (i.~\!e., the Polya-Szeg\"o inequality~\cite{PS45}), also in the fractional framework the energy of $u^*$ decreases with respect to that of $u$. Again, by using the Fourier characterization of $[u]_{H^s(\R^N)}$ given by Proposition~\ref{pro_equiv}, one can plainly apply the symmetrization lemma by Beckner (\cite{Bec92}; see also~\cite{AL89}) to obtain the following
\begin{lemma}\label{lem_pz}{\rm (see, e.g.,} \cite[Theorem 1.1]{Par11}{\rm)}.
Let $s\in (0,1)$. For any $u\in H^s(\R^N)$, the following inequality holds
\begin{equation}\label{eq_pz}
\dys
\iint_{\R^N\times\R^N} \frac{|u^*(x)-u^*(y)|^2}{|x-y|^{N+2s}} dx\, dy
\, \leq \, \iint_{\R^N\times\R^N} \frac{|u(x)-u(y)|^2}{|x-y|^{N+2s}} dx\, dy,
\end{equation}
where $u^*$ denotes the symmetric radial decreasing rearrangement of $u$.
\end{lemma}

\vspace{2mm}
Next we recall two results which we will use in the proof of Theorem~\ref{teorema} (see, in particular,  Step 2 there).
The first one is the following \emph{radial lemma}. 
\begin{lemma}\label{lem_radial} 
Let $u\in L^2(\R^N)$ be a nonnegative radial decreasing  function. 
Then 
\begin{equation*}
|u(x)| \leq  \left(\frac{N}{\omega _{N-1}}\right)^{1/2} |x|^{-N/2} \|u \|_{L^2(\R^N)}, 
\quad\forall x\neq 0,   
\end{equation*}
where $\omega_{N-1}$ is the Lebesgue measure of the unit sphere in $\R^{N}$. 
\end{lemma}

\begin{proof}
Setting $r=|x|$, we have that, for every $r>0$, 
\begin{equation*}
 \|u \|^2_{L^2(\R^N)} \, = \, \int_{\R^N} |u(x)|^2 \, dx 
                                 \, \geq \, \omega_{N-1} \int_0^R  |u(r)|^2 r^{N-1} \, dr 
                                 \, \geq  \, \omega_{N-1} |u(R)|^2 \frac{R^N}{N},  \nonumber
\end{equation*}
where in the last inequality we used the fact that $u$ is decreasing. 
\end{proof} 
The second result is a \emph{compactness lemma} due to Strauss \cite{Str77}
(see also~\cite[Theorem A.I]{BL83} for a simple proof).

\begin{lemma}\label{lem_compact} 
Let $P, Q:\R\rightarrow\R$ be two continuous
functions satisfying 
\begin{equation}
\frac{P(t)}{Q(t)} \rightarrow 0, \quad \mathrm{as\ } |t|\rightarrow +\infty . \label{PQ}
\end{equation}
Let $u_n:\R^N\rightarrow\R$ be a sequence of measurable functions such that 
\begin{equation}
\sup_n \int_{\R^N}  |Q(u_n(x))| \, dx < +\infty,   \label{PQ1}
\end{equation}
and 
\begin{equation}
P(u_n(x))\rightarrow v(x) \quad \mathrm{a.~\!e.~in}~\R^N \quad \mathrm{as\ } n\rightarrow +\infty. 
\label{PQ2}
\end{equation} 
Then, for every bounded Borel set $B$, we have
\begin{equation}
\int_B  |P(u_n(x)) - v(x)| \, dx\rightarrow 0 \quad \mathrm{as\ } n\rightarrow +\infty. 
\label{conv}
\end{equation}
If we further assume that 
\begin{equation}
\frac{P(t)}{Q(t)} \rightarrow 0 \quad \mathrm{as\ } t\rightarrow 0, 
\label{PQ4}
\end{equation}
and 
\begin{equation}
u_n(x) \rightarrow 0 \quad \mathrm{as\ } |x|\rightarrow +\infty, \quad 
\mathrm{uniformly\ with\ respect\ to\ } n,  \label{PQ5}
\end{equation}
then $P(u_n)$ converges to $v$ in $L^1(\R^N)$ as $n\rightarrow +\infty$. 
\end{lemma}

We conclude this section with the following Lemma
\ref{lem_wR}, in which we state and prove some $H^s$ estimates, 
which, in turn,
imply that there exists a nontrivial  competitor for the variational problem~\rif{var}, as described in the subsequent Remark~\ref{rem_notrivial}.

\begin{lemma}\label{lem_wR}
Let~$\zeta$, $R>0$. For any~$t\ge 0$ let
$$ v_R(t):=\left\{
\begin{matrix}
\zeta & {\mbox{ if $t\in [0,R]$,}}\\[1ex]
\zeta \,(R+1-t)  & {\mbox{ if $t\in (R,R+1)$,}}\\[1ex]
0 & {\mbox{ if $t\in [R+1,+\infty)$.}}
\end{matrix}
\right.$$
For any~$x\in \R^N$, let~$w_R(x):=v_R(|x|)$. 

Then, $w_R\in H^s(\R^N)$
for any~$s\in(0,1)$ and there exists~$C(N,s,R)>0$ such 
that~$\|w_R\|_{H^s(\R^N)}\le C(N,s,R)\,\zeta$.
\end{lemma}

\begin{proof} Clearly,
\begin{equation}\label{s3}
\|w_R\|^2_{L^2(\R^N)} \, \le \,  \int_{B_{R+1}} \zeta^2 \,dx \, =  \,  \omega_N (R+1)^N 
\zeta^2.
\end{equation}
Now, we let
$$ \sigma:=\left\{
\begin{matrix}
2s & {\mbox{ if $s\in(0,\, 1/2)$,}}\\
1/2 & {\mbox{ if $s=1/2$,}}
\\ 2s -1 & {\mbox{ if $s\in(1/2,\,1)$.}}
\end{matrix}
\right. $$
We remark that~$\sigma\in (0,1)$ and therefore, by \cite[Lemma~13]{CV},
\begin{equation}\label{cv}
\iint_{B_R\times (\R^N\setminus B_R)}\frac{dx\,dy}{|x-y|^{N+\sigma}}
 \, \le \,  C_1(N,\sigma) \,R^{N-\sigma}.
\end{equation}
Furthermore, if~$x\in B_{R+1}\setminus B_R$ and~$y\in B_R$, we have that
$$
|x-y|\ge |x|-|y|\ge |x|-R=\big||x|-R\big|
$$
and
$$
1=(R+1)-R\ge |x|-R=\big||x|-R\big|;
$$
hence
$$
\big||x|-R\big|^2 
 \, \le  \, \min \big\{ 1, \,  |x-y|^2 \big\}
 \, \le \, \min \big\{ 1, \,  |x-y|, |x-y|^{1/2} \big\}.
$$
As a consequence,
if~$x\in B_{R+1}\setminus B_R$ and~$y\in B_R$,
\begin{eqnarray*}
\frac{\big||x|-R\big|^2}{|x-y|^{N+2s}}
&\le& \left\{
\begin{matrix}
1\cdot |x-y|^{-N-2s} & {\mbox{ if $s\in(0,\, 1/2)$,}}\\[1ex]
|x-y|^{1/2}\cdot |x-y|^{-N-2s}
& {\mbox{ if $s=1/2$,}}\\[1ex]
|x-y| \cdot |x-y|^{-N-2s}
&{\mbox{ if $s\in(1/2,\,1)$}}\end{matrix}\right.
\\[1ex]
 &=& |x-y|^{-N-\sigma}.
\end{eqnarray*}
Accordingly, by~\eqref{cv},
\begin{equation}\label{s1}
\begin{split}
&
\int_{B_R}\left( 
\int_{B_{R+1}\setminus B_R} \frac{\big||x|-R\big|^2}{|x-y|^{N+2s}}
\,dx\right)\,dy
 \, \le \, 
\int_{B_R}\left(
\int_{B_{R+1}\setminus B_R} |x-y|^{-N-\sigma}
\,dx\right)\!dy
\\ &\qquad \, \le \, 
\int_{B_R}\left(
\int_{\R^N\setminus B_R} |x-y|^{-N-\sigma}\,dx\right)\!dy
 \, \le \,  C_1(N,\sigma)\, R^{N-\sigma}
\end{split}
\end{equation}
for a suitable~$C_1(N,\sigma)>0$. 

\vspace{2mm}

Similarly, if~$x\in \R^N\setminus B_{R+1}$ and~$y\in B_{R+1}\setminus B_R$, we have that
$$
|x-y|  \, \ge \,  |x|-|y|
 \, \ge \,  R+1 - |y|
 \, = \, \big||y|- \left(R+1\right) \big|
$$
and
$$
1 \, = \, (R+1)-R
 \, \ge \,  (R+1) - |y|
 \, = \, \big||y| - \left(R+1\right) \big|;
 $$
hence
$$
\big||y|-\left(R+1\right)\big|^2
 \, \le \, \min \big\{ 1,\, |x-y|^2 \big\}
 \, \le \, \min \big\{ 1,\, |x-y|,\, |x-y|^{1/2} \big\}.
$$ 
Then, if~$x\in \R^N\setminus B_{R+1}$ and~$y\in B_{R+1}\setminus B_R$, 
\begin{eqnarray*}
\frac{\big||y|-\left(R+1\right)\big|^2}{|x-y|^{N+2s}}
\!&\le&\! \left\{
\begin{matrix}
1\cdot |x-y|^{-N-2s} & {\mbox{ if $s\in(0,\, 1/2)$,}}\\[1ex]
|x-y|^{1/2}\cdot |x-y|^{-N-2s}
& {\mbox{ if $s=1/2$,}}\\[1ex]
|x-y| \cdot |x-y|^{-N-2s}
&{\mbox{ if $s\in(1/2,\,1)$}}\end{matrix}\right.
\\[1ex]
&=&\! |x-y|^{-N-\sigma}.
\end{eqnarray*} 
Now, using again~\eqref{cv}, with ~$R+1$ instead of ~$R$, we get
\begin{equation}\label{s7}
\begin{split}
& \int_{B_{R+1}\setminus B_R}\left( 
\int_{\R^N\setminus B_{R+1}} \frac{\big||y|-\left(R+1\right)\big|^2}{|x-y|^{N+2s}}
\,dx\right)\,dy
\\[1ex]
 &\qquad \qquad \qquad\le
\int_{B_{R+1}\setminus B_R}\left(
\int_{\R^N\setminus B_{R+1}} |x-y|^{-N-\sigma}
\,dx\right)\,dy
\\[1ex]
 &\qquad\qquad \qquad\le
\int_{B_{R+1}}\left(
\int_{\R^N\setminus B_{R+1}} |x-y|^{-N-\sigma}\,dx\right)\,dy
\ \le \ C_2(N,\sigma)\, \left(R+1\right)^{N-\sigma},
\end{split}
\end{equation}
for a suitable~$C_2(N,\sigma)>0$. 

\vspace{2mm}

Moreover, if~$x\in \R^N\setminus B_{R+1}$ and~$y\in B_R$, we have that
\begin{eqnarray*}
|x-y| \!  & \ge & \! |x|-|y|
\ = \ \frac{|x|}{R+1}+\frac{R\,|x|}{R+1}-|y|
\frac{|x|}{R+1}+\frac{R(R+1)}{R+1}-R
\\
& \ge &\!\frac{|x|}{R+1},
\end{eqnarray*}
and therefore
\begin{equation}\label{s2}
\begin{split}
& \int_{B_R}\left(\int_{\R^N\setminus B_{R+1}}
\frac{1}{|x-y|^{N+2s}}
\,dx\right)\,dy 
\\[1ex]
&\qquad \ \le\, (R+1)^{N+2s} \int_{B_R}\left(\int_{\R^N\setminus 
B_{R+1}}
\frac{1}{|x|^{N+2s}}
\,dx\right)\,dy
\\[1ex]
&\qquad \ =\,\omega_N \omega_{N-1} R^N
(R+1)^{N+2s} \int_{R+1}^{+\infty}
\frac{\varrho^{N-1} }{\varrho^{N+2s}}
\,d\varrho \ \le \ C_3(N,s)\, (R+1)^{2N} 
\end{split}
\end{equation}
for a suitable~$C_3(N,s)>0$. 

\vspace{1mm}
Now, we observe that if~$x$, $y\in B_{R+1}\setminus B_R$,
we have that~$|x-y|\le 2(R+1)$. Thus,
we make the substitution~$z:=x-y$ in the following computation
\begin{equation}\label{s6}\begin{split}
&\iint_{(B_{R+1}\setminus B_R)\times(B_{R+1}\setminus B_R)}
\frac{\big||x|-|y|\big|^2}{|x-y|^{N+2s}}\,dx\,dy \\[1ex]
&\qquad \qquad\qquad\qquad\le \, \iint_{(B_{R+1}\setminus B_R)\times(B_{R+1}\setminus B_R)} 
|x-y|^{2-N-2s}\,dx\,dy \\[1ex]
&\qquad\qquad\qquad \qquad\le \,  \int_{ B_{R+1}\setminus B_R }
\left(
\int_{B_{2(R+1)}} |z|^{2-N-2s}
\,dz\right)\,dx
\\[1ex]
 &\qquad\qquad\qquad\qquad\le \, \omega_N \omega_{N-1} (R+1)^N \int_{0}^{2(R+1)}
\varrho^{2-N-2s}\varrho^{N-1}\,d\varrho
\\[1ex]
 &\qquad\qquad\qquad\qquad\le \, C_4(N,s) \,(R+1)^{N+2-2s},
\end{split}\end{equation}
for a suitable~$C_4(N,s)>0$.

Thus, making use of~\eqref{s1}, \eqref{s7}, \eqref{s2}
and~\eqref{s6}, we conclude that
\begin{eqnarray*}
&&\!\!\!\!\!\!\!\!\!\!\!\!\!\!\!\!\!\!
\iint_{\R^N\times\R^N}\frac{|w_R(x)-w_R(y)|^2}{|x-y|^{N+2s}}\,dx\,dy\\[2ex]
&=& 
\iint_{B_{R+1}\times 
B_{R+1}}\frac{|w_R(x)-w_R(y)|^2}{|x-y|^{N+2s}}\,dx\,dy \\
&&\qquad+
2\iint_{B_{R+1}\times(\R^N\setminus
B_{R+1})}\frac{|w_R(x)-w_R(y)|^2}{|x-y|^{N+2s}}\,dx\,dy
\\[2ex] &=&
2\iint_{B_{R}\times
(B_{R+1}\setminus B_R)}\frac{|w_R(x)-w_R(y)|^2}{|x-y|^{N+2s}}\,dx\,dy 
\\ &&\qquad+
\iint_{(B_{R+1}\setminus B_R)\times
(B_{R+1}\setminus B_R)}\frac{|w_R(x)-w_R(y)|^2}{|x-y|^{N+2s}}\,dx\,dy
\\ &&\qquad+
2\iint_{B_{R+1}\times(\R^N\setminus
B_{R+1})}\frac{|w_R(x)-w_R(y)|^2}{|x-y|^{N+2s}}\,dx\,dy
\\[2ex] &=&
2\int_{B_R}\left(\int_{B_{R+1}\setminus B_R} 
\frac{\zeta^2\big|R-|x|\big|^2}{|x-y|^{N+2s}}
\,dx
\right)\,dy\\ 
&&\qquad+
\iint_{(B_{R+1}\setminus B_R)\times
(B_{R+1}\setminus 
B_R)}\frac{\zeta^2\big| |x|-|y|\big|^2}{|x-y|^{N+2s}}\,dx\,dy
\\ &&\qquad+
2\int_{B_R}\left(\int_{\R^N\setminus B_{R+1}}
\frac{\zeta^2}{|x-y|^{N+2s}}
\,dx
\right)\,dy \\
\\ &&\qquad+
2\int_{B_{R+1}\setminus B_R}\left(\int_{\R^N\setminus B_{R+1}}
\frac{\zeta^2\big| |y|-\left( R+1\right)\big|^2}{|x-y|^{N+2s}}
\,dx
\right)\,dy
\\[2ex] &\le& 2\zeta^2\Big(  C_1(N,\sigma) \,R^{N-\sigma}+ C_4(N,s) 
\,(R+1)^{N+2-2s}+C_3(N,s) 
\,(R+1)^{2N} \\
&&\qquad \, +\,C_2(N,\sigma) 
\,(R+1)^{N-\sigma}\Big).
\end{eqnarray*}
{F}rom this and~\eqref{s3}, the desired result easily follows.
\end{proof}

\begin{rem}\label{rem_notrivial}
By Lemma \ref{lem_wR},
the set in the minimum problem~\rif{var}  
is not empty. Indeed, if $w_R\in H^s(\R^N)$ is defined as in Lemma~\ref{lem_wR}, we have that 
\begin{eqnarray*}
\int_{\R^N} G\left(w_R(x) \right) \, dx  
&=&  \int_{B_{R+1}} G\left(w_R(x) \right) \, dx \nonumber\\
&=& \int_{B_{R}} G\left(w_R(x) \right) \, dx + \int_{B_{R+1}\setminus B_R} G\left(w_R(x) \right) \, dx \nonumber\\
&\geq & G\left(\zeta\right)    |B_R| -  |B_{R+1}\setminus B_R| \left(\max_{t\in\left[0, \zeta\right]} |G(t)|\right), \nonumber
\end{eqnarray*}
where $|\cdot |$ denotes the Lebesgue measure. 
This implies that there exist two positive constants $C_1$ and $C_2$ such that 
\begin{equation*}
\int_{\R^N} G\left(w_R(x) \right) \, dx \, \geq \, C_1 R^N - C_2 R^{N-1}, \nonumber
\end{equation*}
and so we can choose $R>0$ large enough such that 
$\dys \int_{\R^N} G\left(w_R(x) \right) \, dx >0$. 
\vspace{1mm}

Now we make the scale change $w_{R, \sigma}(x)= w_R\left({x}/{\sigma}\right)$, 
and a suitable choice of $\sigma >0$, so that 
$$
\dys \int_{\R^N} G\left(w_{R, \sigma}(x) \right) \, dx
\,  =  \, \sigma^N 
\int_{\R^N} G\left(w_{R}(x) \right) \, dx =1.
$$  
\end{rem}

\vspace{2mm}
\section{Proof of Theorem~\ref{teorema}}\label{sec_main}
In the same spirit of the proof of Theorem~2 in~\cite{BL83}, we divide that of Theorem~\ref{teorema} in a few steps. For the reader's convenience, we will give full details of the proof,  by taking into account the preliminary results in Section~\ref{sec_tools} together with the modifications due to the presence of the fractional Sobolev spaces.

\begin{proof}

\noindent
\\ {\it Step 1 - A minimizing sequence $u_n$.\ } 
Consider a sequence $\left\lbrace u_n\right\rbrace\subseteq H^s(\R^N)$ 
such that $\dys \int_{\R^N} G(u_n)\, dx = 1$ and 
\begin{equation}
\lim_{n\rightarrow +\infty} [u_n]^2_{H^s(\R^N)} = \inf \left\{ [u]^2_{H^s(\R^N)} : u\in H^s(\R^N), \int_{\R^N} G(u)\, dx = 1 \right\}
\, \geq \, 0.    \label{lim}
\end{equation}
By triangle inequality,
$$
\big||u_n(x)|-|u_n(y)|\big| \leq |u_n(x)-u_n(y)|$$
thus the Gagliardo semi-norm of $|u_n|$
is not bigger than the one of $u_n$.

So, without loss of generality, we may suppose that $u_n$ is
nonnegative.
Let $u^{*}_n$ denote the symmetric radial decreasing rearrangement of $u_n$. 
Then
$$\dys \int_{\R^N} G(u^{*}_n)\, dx = \int_{\R^N} G(u_n)\, dx = 1$$
and so,
in view of Lemma~\ref{lem_pz},
we have that
$\{u^{*}_n\}$ is also a minimizing sequence.

These observations imply that we can select a sequence $\{u_n\}$ in such a way that, 
for every $n\in\N$, $u_n$ is nonnegative, spherically symmetric 
and decreasing in $r=|x|$.

\noindent
\\ {\it Step 2 - A priori estimates for $u_n$.\ }
We want to obtain bounds uniform in $n$ on
$\|u_n \|_{L^{q}(\R^N)}$, for every $2 \leq q \leq {2N}/{(N-2s)}$,
and on
$\|u_n\|_{H^{s}(\R^N)}$.

We begin with $\|u_n\|_{H^{s}(\R^N)}$.
Clearly, by $(\ref{lim})$,  $[u_n]^2_{H^s(\R^N)}\leq C$ 
for some positive constant $C$ (recall also
Remark \ref{rem_notrivial}).
Therefore,  it remains to prove that $\|u_n\|_{L^{2}(\R^N)}$ is bounded. 
To do this, we set
$$
g_1(t) := |t|^{p-1}t,\qquad g_2(t) := t,\qquad G_1(t):=\frac1{p+1}|t|^{p+1}
\qquad{\mbox{and}}\qquad G_2(t):=\frac12|t|^2.
$$
Then
$$
g(t)=g_1(t) - g_2(t),
$$ 
and so
\begin{equation}
 G(z) \, = \, \int_{0}^{z} g(t) dt 
                       \, = \,\int_{0}^{z} g_1(t)dt - \int_{0}^{z} g_2(t) dt
                       \, = \,G_1 (z) - G_2 \left(z\right), \quad \forall z\geq 0.  
\end{equation}
Since $p<{(N+2s)}/{(N-2s)}$, we have that for every $\epsilon >0$
there exists a positive constant $C_\epsilon$ such that 
$g_1 (t) \leq C_\epsilon |t|^{\frac{N+2s}{N-2s}} + \epsilon g_2 (t)$. 
This implies that $G_1 (z) \leq C_\epsilon |z|^{\frac{2N}{N-2s}} + \epsilon G_2 (z)$. 
Choosing $\epsilon ={1}/{2}$, we get 
\begin{equation}
G_1 (z) \, \leq \, C |z|^{\frac{2N}{N-2s}} + \frac{1}{2} G_2 (z).  \label{G1}
\end{equation}
Now, the condition $\dys \int_{\R^N} G(u_n)\, dx = 1$ can be written in the following form 
\begin{equation}
\int_{\R^N} G_1 (u_n)\, dx \, = \, \int_{\R^N} G_2 (u_n)\, dx +1.    \label{G2}
\end{equation}
Putting together $(\ref{G1})$ and $(\ref{G2})$, we obtain 
\begin{equation}
\frac{1}{2} \int_{\R^N} G_2 (u_n)\, dx  + 1
\, \leq \, C   \int_{\R^N} |u_n|^{\frac{2N}{N-2s}}\, dx.  \label{G3}
\end{equation}
Now we use the fractional Sobolev embedding theorem (see, e.g., \cite[Theorem 6.5]{DPV12}) to say that 
$$
\|u_n \|_{L^{\frac{2N}{N-2s}}(\R^N)}
\,\leq\, C [u_n]_{H^s(\R^N)},
$$
where the constant $C$ does not depend on $n$. Thus, since $u_n$ is a minimizing sequence, the boundedness of $[u_n]^2_{H^s(\R^N)}$ yields that of $\|u_n\|_{L^{\frac{2N}{N-2s}}({\R^N})}$. 
By the definition of $G_2$, the inequality in~$\rif{G3}$ implies that 
\begin{equation*}
\frac{1}{2} \int_{\R^N} u_n^2 \, dx \, = \, \int_{\R^N} G_2 (u_n)\, dx \ \leq \ C, 
\end{equation*}
and thus we bound $\|u_n\|^2_{L^2(\R^N)}$
(and so $\|u_n\|^2_{H^{s}(\R^N)}$) uniformly in $n$.

Finally, 
by the bounds on $\|u_n\|^2_{L^2(\R^N)}$
and $\|u_n \|_{L^{\frac{2N}{N-2s}}(\R^N)}$, using the
H\"{o}lder inequality, we obtain that
$\|u_n \|_{L^{q}(\R^N)}\leq C$ for every $2 \leq q \leq {2N}/{(N-2s)}$.

\noindent
\\ {\it Step 3 - Passage to the limit and conclusion of the proof.\ } Since $u_n \in L^2(\R^N)$ is a sequence of nonnegative radial decreasing functions, 
we can apply Lemma \ref{lem_radial}
to get 
\begin{equation}
|u_n(x)| \leq  \left(\frac{N}{\omega _{N-1}}\right)^{\!\frac{1}{2}} |x|^{-N/2} \|u_n \|_{L^2(\R^N)}. 
\end{equation}
From the previous step we have that $u_n$ is uniformly bounded in $L^2(\R^N)$; 
then $|u_n(x)| \leq C |x|^{-N/2}$, with $C$ independent of $n$. 
This implies that $u_n(x) \rightarrow 0$ as $|x|\rightarrow +\infty$ 
uniformly with respect to $n$. 
Now, since $u_n$ is bounded in $H^{s}(\R^N)$, 
we can extract a subsequence of $u_n$, 
again denoted by $u_n$, 
such that $u_n$ converges weakly in $H^{s}(\R^N)$ and almost everywhere in $\R^N$ 
to a function $\overline{u}$. 
Moreover, by construction, $\overline{u}\in H^{s}(\R^N)$ is spherically symmetric and decreasing in $r$. 

\vspace{1mm} 
Now, in order to apply Lemma \ref{lem_compact} (with $P:=G_1$), 
consider the polynomial function $Q$ defined by
$$
Q(t):= t^2 + |t|^{\frac{2N}{N-2s}}.
$$ 
Since the sequence $u_n$ is uniformly bounded in $L^2(\R^N)$ and in $L^{\frac{2N}{N-2s}}(\R^N)$, 
we have that $Q$ satisfies 
\begin{equation*}
\int_{\R^N} |Q(u_n(x))| \, dx \, = \,  
\int_{\R^N} \left( u_n^2(x)  + |u_n(x) |^{\frac{2N}{N-2s}}\right)  \, dx
 \,  \leq  \, C, \quad 
\mathrm{for\ every\ } n\in\N. \nonumber
\end{equation*} 
Moreover, if $G_1$ is defined as in the previous step, 
by the fact that $p\in\left(1, \frac{N+2s}{N-2s}\right)$ we derive
\begin{equation*}
\frac{G_1(t)}{Q(t)}\rightarrow 0, \quad  \mathrm{as \ } 
t\rightarrow +\infty   \mathrm{\ and\ } t\rightarrow 0.  \nonumber
\end{equation*}
Since $u_n$ converges almost everywhere in $\R^N$ to $\overline{u}$, 
we have that also $G_1\left(u_n\right)$ converges $G_1\left(\overline{u}\right)$. 
Finally, $u_n(x) \rightarrow 0$ as $|x|\rightarrow +\infty$ 
uniformly with respect to $n$. 
Therefore Lemma \ref{lem_compact} holds, getting 
\begin{equation*}
\int_{\R^N} G_1\left(u_n(x) \right) \, dx \, \rightarrow \, 
\int_{\R^N} G_1\left(\overline{u} (x) \right) \, dx \quad \mathrm{as\ } n\rightarrow +\infty. \nonumber
\end{equation*}
Thus, using Fatou's Lemma in $(\ref{G2})$, we obtain that 
\begin{equation}\label{3Z}
\int_{\R^N} G_1\left(\overline{u} (x) \right) \, dx \geq \int_{\R^N} G_2\left(\overline{u} (x) \right) \, dx + 1,  
\end{equation}
that is $$\int_{\R^N} G\left(\overline{u} (x) \right) \, dx \geq 1 .$$
On the other hand, using again Fatou's Lemma, we have that
\begin{equation}\label{3A}\begin{split}
&[\overline{u}]^2_{H^s(\R^N)} \, \leq  \, \lim_{n\rightarrow +\infty} [u_n]^2_{H^s(\R^N)} 
\\ &\qquad= \inf \left\{ [u]^2_{H^s(\R^N)} : u\in H^s(\R^N), \int_{\R^N} G(u)\, dx = 1 \right\}. 
\end{split}\end{equation}
Now, suppose by contradiction that $\dys \int_{\R^N} G(\overline{u} (x)) \, dx > 1$. 
Then, by the scale change $\overline{u}_\sigma (x) = \overline{u}({x}/{\sigma})$, we have 
\begin{equation}\label{3B}
\int_{\R^N} G\left(\overline{u}_\sigma (x) \right) \, dx
\,  =\, \sigma^N 
\int_{\R^N} G\left(\overline{u} (x) \right) \, dx 
\, = \, 1 
\end{equation}
for some \begin{equation}\label{3C}
\sigma\in\left(0, 1\right).\end{equation}
Moreover, we have 
\begin{eqnarray}
[\overline{u}_\sigma]^2_{H^s(\R^N)}\! &=&\! \sigma^{N-2s} [\overline{u}]^2_{H^s(\R^N)} \nonumber\\
&\leq & \! \sigma^{N-2s} \inf \left\{ [u]^2_{H^s(\R^N)} : u\in H^s(\R^N), \int_{\R^N} G(u)\, dx = 1 \right\},
\end{eqnarray} 
due to \eqref{3A}, and 
\begin{equation}
\inf \left\{ [u]^2_{H^s(\R^N)} : u\in H^s(\R^N), \int_{\R^N} G(u)\, dx = 1 \right\}
\,  \leq  \, [\overline{u}_\sigma]^2_{H^s(\R^N)}, \nonumber
\end{equation}
thanks to \eqref{3B}.
Combining the last two inequalities and recalling \eqref{3C}, we get 
$$
\dys
\inf \left\{ [u]^2_{H^s(\R^N)} : u\in H^s(\R^N), \,\int_{R^N} G(u)\, dx = 1 \right\}=0,
$$ 
hence also $[\overline{u}]^2_{H^s(\R^N)}=0$. 
Then $u\equiv0$, which is in contradiction with \eqref{3Z}.
Therefore,
$
\dys \int_{\R^N} G\left(\overline{u} (x) \right) \, dx =1
$
and 
$[\overline{u}]_{H^s(\R^N)}=\inf \big\{ [u]_{H^s(\R^N)} : u\in H^s(\R^N), \int G(u)\, dx = 1 \big\}$;  
that is, $\overline{u}$ solves the minimization problem $(\ref{var})$.
\end{proof}

\vspace{2mm}

\end{document}